\documentclass[12pt,reqno]{amsart}

\usepackage{amssymb,latexsym}

\usepackage{enumerate}
\allowdisplaybreaks
\usepackage[french,english]{babel}
\usepackage{amsmath}
\usepackage{graphicx}
\usepackage{amssymb}
\usepackage{bbm}
\usepackage{amsthm,mathtools}
\usepackage{ulem}
\usepackage{geometry}
\usepackage{tikz-cd}
\usepackage{mathrsfs}
\usepackage[colorinlistoftodos]{todonotes}
\usepackage{enumitem}
\usepackage{verbatim}
\usepackage[foot]{amsaddr}
\usepackage{dsfont}
\usepackage{cite}
\usepackage[T1]{fontenc}

\makeatletter

\@namedef{subjclassname@2010}{
	
	\textup{2020} Mathematics Subject Classification}

\makeatother
\newtheorem{thm}{Theorem}[section]
\newtheorem*{thm*}{Theorem}

\newtheorem{lem}[thm]{Lemma}

\theoremstyle{definition}

\numberwithin{equation}{section}

\newcommand{\bg}{\big}

\newcommand{\bgg}{\bigg}

\newcommand{\Bg}{\Big}

\newcommand{\lo}{\log_2}
\newcommand{\lt}{\log_3}

\newcommand{\mbz}{\mathbb{Z}}
\newcommand{\mbk}{\mathbb{K}}
\newcommand{\mbr}{\mathbb{R}}

\newcommand{\mbq}{\mathbb{Q}}

\newcommand{\mch}{\mathcal{H}}

\newcommand{\mcz}{\mathcal{Z}}

\newcommand{\mcm}{\mathcal{M}}

\newcommand{\mmd}{\mathrm{d}}
\newcommand{\mme}{\mathrm{e}}
\newcommand{\mmi}{\mathrm{i}}
\newcommand{\mit}{\mathrm{i}t}

\newcommand{\whp}{\widehat{\Phi}}

\usepackage{hyperref}
\hypersetup{colorlinks=true,linkcolor=blue,anchorcolor=blue,citecolor=blue}
\usepackage{color}
\newcommand{\newabstract}[1]{%
	\par\bigskip
	\csname otherlanguage*\endcsname{#1}%
	\csname captions#1\endcsname
	\item[\hskip\labelsep\scshape\abstractname.]
}

\begin{document}

	\baselineskip=17pt

	\title[Z. Dong, Q. Yang and S. Zhao]{Extreme values of Dirichlet $L$-functions on the critical line with a growing prime modulus}
	
    \author{Zikang Dong\textsuperscript{1}}
	\author{Qiyu Yang\textsuperscript{2}}
    \author{Shengbo Zhao\textsuperscript{3}}

    \address{1. School of Mathematical Sciences, Soochow University, Suzhou 215006, P. R. China}
	\address{2. School of Mathematics and Statistics, Henan Normal University, Xinxiang 453007, CHINA}
	\address{3. School of Mathematical Sciences, Key Laboratory of Intelligent Computing and Applications (Tongji University), Ministry of Education, Tongji University, Shanghai 200092, China}

    \email{zikangdong@gmail.com}
	\email{qyyang.must@gmail.com}
	\email{shengbozhao@hotmail.com}

	\begin{abstract} 
	   In this paper, we obtain extreme values of Dirichlet $L$-functions on the critical line with a growing prime modulus. Our result yields a larger lower bound for extreme values than that from Soundararajan's classical resonance method. Consequently, we show that one branch of the dichotomy established by Bondarenko, Darbar, Hagen, Heap, and Seip (2023) can be realized for this family of Dirichlet $L$-functions. The proof relies on the resonance method together with precise estimates for GCD sums.
	\end{abstract}
	
	
	
	\maketitle

\section{Introduction}
The study of extreme values of Dirichlet \(L\)-functions is a classical problem in analytic number theory. A significant development was made by Soundararajan \cite{Soundararajan2008MathAnn}, who introduced the resonance method and proved that
\[\max_{T\le t \le 2T}\Bg|\zeta\Bg(\frac{1}{2}+\mit\Bg)\Bg| \ge \exp \bgg(\bg(1+o(1)\bg)\sqrt{\frac{\log T}{\lo T}}\bgg).\]
Here and throughout this paper, we write \(\log_j\) for the \(j\)-th iterated logarithm, such as \(\lo x :=\log\log x\) and \(\lt x := \log\log\log x\). As observed in \cite{Bondarenko2023BLMS}, the same resonance method also furnishes a lower bound of this scale for Dirichlet \(L\)-functions along the critical line. Subsequently, the development of long resonators and their connection with GCD sums led to substantially stronger lower bounds for the Riemann zeta function. In particular, the works of Bondarenko and Seip \cite{Bondarenko2017DukeMJ,Bondarenko2018MathAnn} and de la Bretèche and Tenenbaum \cite{delaBreteche2019PLMS} produced an additional factor \(\sqrt{\lt T}\) in the exponent. A crucial feature of these arguments is the positivity of the relevant Dirichlet coefficients, which is no longer available for a general Dirichlet \(L\)-function.

In 2023, Bondarenko, Darbar, Hagen, Heap and Seip \cite{Bondarenko2023BLMS} approached this difficulty through Dedekind zeta functions. Let \(\mbk=\mbq(\omega_q)\) be a cyclotomic field. They proved, uniformly for \(q \ll (\lo T)^A\), that
\[\max_{0\le|t|\le T}\Bg|\zeta_\mbk\Bg(\frac{1}{2}+\mit\Bg)\Bg|\ge\exp\bgg(\bg(1+o(1)\bg)\sqrt{\phi(q)}\sqrt{\frac{\log T \lt T}{\lo T}}\bgg)\]
for sufficiently large \(T\), where \(A\) is an arbitrary positive number and \(\phi(n)\) denotes Euler’s totient function. The factorization
\begin{align}
    \label{factorization}
    \zeta_\mbk(s) = \zeta(s)\prod_{\substack{\chi(\bmod q)\\ \chi \neq \chi_0}}L(s,\chi^\prime),
\end{align}
where \(\chi^\prime\) denotes the primitive character inducing \(\chi\), then leads to the following dichotomy: either there exists a non-principal Dirichlet character \(\chi\) for which
\[\max_{0\le t\le T}\Bg|L\Bg(\frac{1}{2}+\mit,\chi\Bg)\Bg| \ge \exp\bgg(c\sqrt{\frac{\log T \lt T}{\lo T}}\bgg)\]
for some sufficiently small constant \(c>0\), or the Riemann zeta function must attain still larger values on the critical line. This leaves open the question of whether the Dirichlet \(L\)-function alternative itself can be established directly.

In the present paper, we give such a direct result for a class of growing prime moduli. Our main result is as follows.

\begin{thm}
    \label{thm1}
    Let \(0<\beta<1\), and let \(q=q(T) \le T\) be a prime satisfying
    \[\frac{\log q}{\log T} \to \theta,\qquad 0<\theta \le 1.\]
    Then, for every fixed
    \[0<c<\sqrt{2\min\{\theta,1-\beta\}},\]
    we have
    \[\max_{\substack{\chi(\bmod q) \\ \chi \neq\chi_0}}\max_{T^\beta \le t \le T}\Bg|L\Bg(\frac{1}{2}+\mit,\chi\Bg)\Bg| \ge \exp\bgg(c\sqrt{\frac{\log T \lt T}{\lo T}}\bgg)\]
    for all sufficiently large \(T\).
\end{thm}

It is worth emphasizing that our result differs essentially from the dichotomy of \cite{Bondarenko2023BLMS} in its logical form. Their argument does not directly establish that a Dirichlet \(L\)-function attains values with an additional factor \(\sqrt{\lt T}\) in the exponent compared with the bound furnished by Soundararajan's resonance method \cite{Soundararajan2008MathAnn}; rather, through \eqref{factorization}, it relates this possibility to the occurrence of still larger values of the Riemann zeta function. By contrast, Theorem \ref{thm1} gives an unconditional lower bound of precisely this scale for Dirichlet \(L\)-functions when
\[q=T^{\theta+o(1)}.\]

However, there is an important difference in the range of the modulus. The result of \cite{Bondarenko2023BLMS} is uniform for \(q \ll (\lo T)^A\), where \(A\) is an arbitrary positive number, whereas Theorem \ref{thm1} requires \(q=T^{\theta+o(1)}\) with \(\theta>0\). This restriction reflects the present method. We take a resonator of length \(N=T^{\kappa+o(1)}\). The limitation to the interval \(T^\beta\le t \le T\) requires \(\kappa<1-\beta\), while the error arising from the removal of the principal character requires \(\kappa<\theta\). The latter restriction is the main obstacle to reaching small moduli. More precisely, we complete the average over non-principal characters to the full character family and then subtract the principal character, whose contribution is presently estimated only in absolute value. Extending the argument to \(q=T^{o(1)}\), and in particular to the range \(q\ll(\lo T)^A\), would therefore require a more refined separation and treatment of the principal and non-principal character contributions. In addition, the resonator construction would have to be adapted, since the present argument uses the fact that every element of \(\mcm\) is coprime to \(q\).

\section{Proof of Theorem \ref{thm1}}
\label{sec-pthm1}

\subsection{Auxiliary lemmas}

In this section, we present several lemmas that will play an important role in the subsequent proofs. We begin with the classical convexity bound for Dirichlet $L$-functions.
\begin{lem}
    \label{lem1}
    Let \(q\) be a prime and let \(\chi\) be a non-principal Dirichlet character modulo \(q\). Then, for every real \(t\), we have
    \[L\Bg(\frac{1}{2}+\mit,\chi\Bg) \ll (q(1+|t|))^{1/4}.\]
\end{lem}
\begin{proof}
    Since \(q\) is prime and \(\chi\) is non-principal, \(\chi\) is primitive. Therefore, the result follows from \cite[Theorem 5.23]{Iwaniec2004book}.
\end{proof}

Throughout the remainder of the paper, for \(f \in L^1(\mbr)\), we use the Fourier transform convention
\[\widehat{f}(\xi) := \int_\mbr f(x) \mme^{-\mmi\xi x} \mmd x.\]
Inspired by \cite{delaBreteche2019PLMS,Ddyang2022Mathematika}, we establish the following double-version convolution formula.

\begin{lem}
    \label{lem2}
    Let \(\chi\) be a non-principal Dirichlet character modulo \(q\). Let \(a>1/2\) and set \(z=x+\mmi y\). Suppose that \(K(z)\) is holomorphic in the horizontal strip \(-a\le y \le 0\) satisfying the growth condition
    \[\max_{-a\le y \le 0}|K(x+\mmi y)| = O\Bg(\frac{1}{(1+|x|)^2}\Bg)\]
    as \(|x| \to \infty\). Then, we have, for every real \(t\),
    \begin{align*}
        \int_\mbr L\Bg(\frac{1}{2}+\mit+\mmi u,\chi \Bg)L\Bg(\frac{1}{2}-\mit+\mmi u,\overline{\chi} \Bg)K(u)\mmd u = \sum_{k,\ell \ge1}\frac{\chi(k)\overline{\chi}(\ell)}{\sqrt{k\ell}}\Bg(\frac{\ell}{k}\Bg)^{\mit}\widehat{K}(\log k\ell).
    \end{align*}
\end{lem}

\begin{proof}
Put \(G(z):= L(1/2+\mit+z,\chi)L(1/2-\mit+z,\overline{\chi})K(-\mmi z)\). Since \(\chi\) is non-principal, both \(L(s,\chi)\) and \(L(s,\overline{\chi})\) are entire. Hence \(G(z)\) is holomorphic in the strip \(0\leq \Re z\leq a\). Let \(Y>2|t|+2\), and consider the rectangle with vertices \(\pm\mmi Y\) and \(a\pm\mmi Y\). By Cauchy's theorem,
\[\Bg(\int_{-\mmi Y}^{a-\mmi Y}+\int_{a-\mmi Y}^{a+\mmi Y}+\int_{a+\mmi Y}^{\mmi Y}+\int_{\mmi Y}^{-\mmi Y} \Bg) G(z)\mmd z =0.\]

We first show that the integrals over the two horizontal sides tend to zero as \(Y \to \infty\). By Lemma \ref{lem1} and the Phragm\'en--Lindel\"of principle, we have, uniformly for $0\leq \sigma\leq a$,
\[L\Bg(\frac12+\sigma+\mmi v,\chi\Bg)\ll_{a,\varepsilon}\bg(q(1+|v|)\bg)^{1/4+\varepsilon}.\]
Combining this with the growth condition on \(K\), we obtain 
\[\int_{-\mmi Y}^{a-\mmi Y}G(z)\mmd z = o(1) \quad \text{and} \quad \int_{a+\mmi Y}^{\mmi Y}G(z)\mmd z = o(1)\]
as \(Y \to \infty\). Thus, letting \(Y \to \infty\) yields that
\[\int_{-\mmi\infty}^{\mmi \infty}G(z)\mmd z = \int_{a-\mmi\infty}^{a+\mmi\infty}G(z)\mmd z.\]
Writing $z=\mmi u$ on the left-hand side and $z=a+\mmi u$ on the
right-hand side, respectively, gives
\begin{align}
\label{lem2-equation}
&\int_\mbr L\Bg(\frac12+\mit+\mmi u,\chi\Bg) L\Bg(\frac12-\mit+\mmi u,\overline{\chi}\Bg)K(u)\mmd u \nonumber\\
&\qquad= \int_\mbr L\Bg(\frac12+a+\mit+\mmi u,\chi\Bg) L\Bg(\frac12+a-\mit+\mmi u,\overline{\chi}\Bg) K(u-\mmi a)\mmd u.
\end{align}

Since \(a>1/2\), both Dirichlet series on the right-hand side of
\eqref{lem2-equation} converge absolutely, so we may integrate term by term.
Moreover, we have
\[\int_\mbr K(u-\mmi a)\mme^{-\mmi u \log k\ell}\mmd u = (k\ell)^a \widehat{K}(\log k\ell).\]
Substituting this into \eqref{lem2-equation} proves this lemma.
\end{proof}

The following lemma provides an upper bound for the second moment of \(L\)-functions over all non-principal characters.

\begin{lem}
    \label{lem3}
    Let $q$ be a prime and let $X\geq 2$. Then
\[\sum_{\substack{\chi(\bmod q)\\ \chi\neq\chi_0}}
\int_{-X}^{X} \Bg|L \Bg(\frac12+\mit,\chi\Bg)\Bg|^2 \mmd t \ll qX\log(qX).\]
\end{lem}

\begin{proof}
    The proof of this result follows by applying \cite[Theorem 3]{Rane1980JLMS} together with some calculations.
\end{proof}

For a finite set \(\mcm\) of positive integers and \(\alpha>0\), define
\[S_\alpha(\mcm) := \sum_{m,n\in\mcm} \Bg(\frac{(m,n)}{[m,n]}\Bg)^\alpha,\]
and write \(S(\mcm) = S_{1/2}(\mcm)\). Here, \((m,n)\) and \([m,n]\) denote the greatest common divisor and least common multiple of \(m\) and \(n\), respectively. The following results play a key role in the proof of Theorem \ref{thm1}.

\begin{lem}
    \label{lem4}
    Let \(\mcm\) be a set of integers with \(|\mcm| =N\). As \(N \to \infty\), we have
    \[\max_{|\mcm|=N}\frac{S(\mcm)}{|\mcm|}=\exp\bigg(\bg(2\sqrt{2}+o(1)\bg)\sqrt{\frac{\log N\log_3N}{\log_2N}}\bigg).\]
    Moreover, \(\mcm\) may be chosen so that
\[y_\mcm := \max_{m\in\mcm}P^+(m) \ll (\log N)^{6/5},\]
where \(P^+(n)\) denotes the largest prime factor of \(n\).
\end{lem}

\begin{proof}
    This is \cite[Eq. (1.4)]{delaBreteche2019PLMS}.
\end{proof}

\subsection{Proof of Theorem \ref{thm1}}

We now turn to the proof of Theorem \ref{thm1}, beginning with the construction of the resonator. Fix \(\kappa\) such that
\[0<\kappa<\min\{\theta,1-\beta\}\]
and put \(N= \lfloor T^\kappa \rfloor\). Choose \(\varepsilon>0\) sufficiently small such that \(\kappa+\varepsilon<\theta\). Let $\mathcal{M}$ be the set of positive integers constructed as in \cite{delaBreteche2019PLMS}, with \(|\mcm|=N\). For sufficiently large \(T\), we have \((m,q)=1\) for every \(m \in \mcm\). For \(j \in \mbz\), define
\[I_j = \bgg[ \Bg(1+ \frac{\log T}{T}\Bg)^j,\Bg(1+ \frac{\log T}{T}\Bg)^{j+1} \bgg).\]
Then, set
\[\mcm_{j,a} := \{m \in \mcm : m \in I_j, m \equiv a \pmod q \}.\]
Whenever \(\mcm_{j,a} \neq \emptyset\), let \(h :=h_{j,a}=\min\mcm_{j,a}\) and 
\[r(h_{j,a}) = \sqrt{|\mcm_{j,a}|} = \sqrt{\sum_{m \in \mcm_{j,a}}}1.\]
Denote by \(\mch\) the set of all \(h\). For every non-principal character \(\chi\) modulo \(q\), define
\[R_\chi(t) : = \sum_{h\in\mch} r(h)\chi(h) h^{-\mit}.\]
By the Cauchy-Schwarz inequality, we have
\begin{align}
    \label{Rupper}
    |R_\chi(t)|^2 \le \Bg(\sum_{h\in\mch}r(h)\Bg)^2 \le N\sum_{h\in\mch}r(h)^2 \le N|\mcm|.
\end{align}
As in \cite{Bondarenko2017DukeMJ}, we take \(\Phi(t) := \mme^{-t^2/2}\), whose Fourier transform satisfies \(\widehat{\Phi}(\xi)=\sqrt{2\pi}\Phi(\xi)>0\) for all \(\xi\in\mbr\). According to \cite[Lemma 5]{Bondarenko2018MathAnn}, we obtain
\begin{align}
    \label{intRPbound}
    \sum_{\substack{\chi(\bmod q) \\ \chi\neq\chi_0} }\int_\mbr |R_\chi(t)|^2 \Phi\Bg(\frac{t\log T}{T}\Bg) \mmd t \ll q \frac{|\mcm|T}{\log T}.
\end{align}
Fix a sufficiently large integer \(\eta \ge 2\), define
\[K(u):= K_\eta(u) = \frac{\sin^{2\eta}(\tfrac{1}{\eta}\varepsilon\log Tu)}{(\tfrac{1}{\eta}\varepsilon\log T)^{2\eta-1}u^{2\eta}}.\]
As in \cite[Lemma 5]{Bondarenko2023BLMS}, its Fourier transform \(\widehat K\) is real, even and non-negative, and is supported on \(|v|\le 2\varepsilon\log T\). Moreover,
\begin{align}
    \label{Kpro1}
    0 \le \widehat{K}(v) \le \widehat{K}(0) \asymp_\eta 1,
\end{align}
and for sufficiently large \(\eta\), we have
\begin{align}
    \label{Kpro2}
    \widehat{K}(v) \ge \frac{1}{2}\widehat{K}(0), \qquad |v| \le \frac{\varepsilon\log T}{3\eta}.
\end{align}
We also have the following bound for \(u\neq0\):
\begin{align}
    \label{Kpro3}
    K(u) \ll_{\eta,\varepsilon} (\log T)^{1-2\eta} |u|^{-2\eta}.
\end{align}

For every non-principal character \(\chi\pmod q\), define
\[\Upsilon_\chi(t,u):=L\Bg(\frac{1}{2}+\mit+\mmi u,\chi \Bg)L\Bg(\frac{1}{2}-\mit+\mmi u,\overline{\chi} \Bg)K(u)\]
and set
\[I(T):=\sum_{\substack{\chi(\bmod q) \\ \chi\neq\chi_0} } \int_\mbr|R_\chi(t)|^2 \Phi\Bg(\frac{t\log T}{T}\Bg)\int_\mbr\Upsilon_\chi(t,u)\mmd u\mmd t.\]
Inspired by \cite{Bondarenko2018MathAnn,Chirre2019JNT,delaBreteche2019PLMS,Ddyang2022Mathematika}, we show that the main contribution to \(I(T)\) comes from the range \(2T^\beta \le |t| \le T/2\) and \(|u| \le |t|/2\). We first consider the range \(|t| \le 2T^\beta\). For \(|u| \le T^\beta\), the Cauchy-Schwarz inequality gives
\begin{align*}
    &\sum_{\substack{\chi(\bmod q) \\ \chi\neq\chi_0} }\int_{|t|\le 2T^\beta}\Bg|L\Bg(\frac{1}{2}+\mit+\mmi u,\chi \Bg)L\Bg(\frac{1}{2}-\mit+\mmi u,\overline{\chi} \Bg)\Bg|\mmd t  \\
    & \quad\le\Bg(\sum_{\substack{\chi(\bmod q) \\ \chi\neq\chi_0} }\int_{|t|\le 2T^\beta} \Bg|L\Bg(\frac{1}{2}+\mit+\mmi u,\chi \Bg)\Bg|^2 \mmd t\Bg)^{1/2} \\
    &\qquad\times\Bg(\sum_{\substack{\chi(\bmod q) \\ \chi\neq\chi_0} }\int_{|t|\le 2T^\beta} \Bg|L\Bg(\frac{1}{2}-\mit+\mmi u,\overline{\chi} \Bg)\Bg|^2 \mmd t\Bg)^{1/2}.
\end{align*}
Since \(|u|\le T^\beta\), both shifted intervals are contained in \([-3T^\beta, 3T^\beta]\). Hence, Lemma \ref{lem3} yields
\[\sum_{\substack{\chi(\bmod q) \\ \chi\neq\chi_0} }\int_{|t|\le 2T^\beta}\Bg|L\Bg(\frac{1}{2}+\mit+\mmi u,\chi \Bg)L\Bg(\frac{1}{2}-\mit+\mmi u,\overline{\chi} \Bg)\Bg|\mmd t \ll qT^\beta \log (qT).\]
Therefore, by the non-negativity of \(K\) and \eqref{Kpro1}, we obtain
\begin{align}
    \label{tail1-1}
    \sum_{\substack{\chi(\bmod q) \\ \chi\neq\chi_0} }\int_{|t|\le 2T^\beta} \int_{|u|\le T^\beta} \Bg|\Upsilon_\chi(t,u)\Bg| \mmd u \mmd t \ll_\eta qT^\beta \log (qT).
\end{align}
For \(|u| > T^\beta\), applying the Cauchy-Schwarz inequality and Lemma \ref{lem3} again,
and combining \eqref{Kpro3}, we obtain
\begin{align}
    \label{tail1-2}
    \sum_{\substack{\chi(\bmod q) \\ \chi\neq\chi_0} }\int_{|t|\le 2T^\beta} \int_{|u|> T^\beta} \Bg|\Upsilon_\chi(t,u)\Bg| \mmd u \mmd t &\ll_{\eta,\varepsilon} q(\log T)^{1-2\eta}\int_{T^\beta}^\infty u^{1-2\eta}\log(q(2+u))\mmd u \nonumber\\
    &\ll_{\eta,\varepsilon} q(\log T)^{2-2\eta}T^{-2\beta(\eta-1)}.
\end{align}
Combining \eqref{tail1-1} and \eqref{tail1-2} with \eqref{Rupper}, and recalling that \(|\mcm|=N\), we obtain
\begin{align}
\label{tail1}
    \Bg|\sum_{\substack{\chi(\bmod q) \\ \chi\neq\chi_0} }\int_{|t|\le 2T^\beta}|R_\chi(t)|^2&\Phi\Bg(\frac{t\log T}{T}\Bg)\int_\mbr\Upsilon_\chi(t,u)\mmd u\mmd t\Bg|\nonumber\\
    &\ll qN^2 T^\beta \log (qT) +qN^2(\log T)^{2-2\eta}T^{-2\beta(\eta-1)} \nonumber\\
    &=O\bg(qN^2 T^\beta \log (qT)\bg) + o\Bg(qN\frac{T}{\log T}\Bg).
\end{align}
Next, consider \(|t|>T/2\). Since \(\eta \ge 2\), Lemma \ref{lem1} and \eqref{Kpro3} imply
\[\int_\mbr\Bg|L\Bg(\frac{1}{2}+\mit+\mmi u,\chi \Bg)L\Bg(\frac{1}{2}-\mit+\mmi u,\overline{\chi}\Bg)\Bg|K(u)\mmd u \ll_{\eta,\varepsilon}q^{1/2+\varepsilon}(1+|t|)^{1/2+\varepsilon}.\]
Together with \eqref{Rupper} and the exponential decay of \(\Phi\), this gives 
\begin{align}
    \label{tail2}
    \Bg|\sum_{\substack{\chi(\bmod q) \\ \chi\neq\chi_0} }\int_{|t|>T/2}|R_\chi(t)|^2 \Phi\Bg(\frac{t\log T}{T}\Bg)\int_\mbr\Upsilon_\chi(t,u)\mmd u\mmd t \Bg|= o\Bg(qN\frac{T}{\log T}\Bg).
\end{align}
Finally, if \(2T^\beta \le |t|\le T/2\), \(|u|>|t|/2\), then \(|u|>T^\beta\) and \(|t|<2|u|\). Thus, using again the Cauchy-Schwarz inequality and Lemma \ref{lem3}, we have
\begin{align}
    \label{tail3}
    \sum_{\substack{\chi(\bmod q) \\ \chi\neq\chi_0} }\int_{2T^\beta \le |t|\le T/2}|R_\chi(t)|^2 &\Phi\Bg(\frac{t\log T}{T}\Bg)\int_{|u|>|t|/2}\Bg|\Upsilon_\chi(t,u)\Bg|\mmd u\mmd t \nonumber \\
    &\ll_{\eta,\varepsilon}qN^2(\log T)^{2-2\eta}T^{-2\beta(\eta-1)} = o\Bg(qN\frac{T}{\log T}\Bg).
\end{align}
Consequently, substituting \eqref{tail1}, \eqref{tail2} and \eqref{tail3} into the definition of \(I(T)\) yields that
\begin{align*}
    I(T) = \sum_{\substack{\chi(\bmod q) \\ \chi\neq\chi_0} }\int_{2T^\beta \le |t|\le T/2}|R_\chi(t)|^2 &\Phi\Bg(\frac{t\log T}{T}\Bg) \int_{|u|\le |t|/2} \Upsilon_\chi(t,u)\mmd u\mmd t \\
    &\quad+O\bg(q N^2 T^\beta \log(qT)\bg) +o\Bg(qN\frac{T}{\log T}\Bg).
\end{align*}
Now, define
\[\mcz_{q,\beta}(T) := \max_{\substack{\chi(\bmod q) \\ \chi\neq\chi_0} }\max_{T^\beta \le t \le T}\Bg|L\Bg(\frac{1}{2}+\mit,\chi\Bg)\Bg|.\]
Note that the region \(2T^\beta \le |t| \le T/2\), \(|u|\le|t|/2\) gives \(T^\beta \le |t\pm u|\le T\). Since \(N=T^{\kappa+o(1)}\) with \(\kappa < 1-\beta\), we have 
\[q N^2 T^\beta \log(qT)=o\Bg(qN\frac{T}{\log T}\Bg),\]
and, by \eqref{intRPbound},
\begin{align}
    \label{max1}
    I(T) \ll qN\frac{T}{\log T}\mcz_{q,\beta}(T)^2 + o\Bg(qN\frac{T}{\log T}\Bg).
\end{align}

Expanding \(|R_\chi(t)|^2\) in \(I(T)\) and using Lemma \ref{lem2}, we have
\[I(T) = \frac{T}{\log T}\sum_{h,g\in\mch}r(h)r(g)\sum_{k,\ell\ge1}\frac{\widehat{K}(\log k\ell)}{\sqrt{k\ell}}\whp\Bg(\frac{T}{\log T}\log\frac{g\ell}{hk}\Bg)\sum_{\substack{\chi(\bmod q) \\ \chi\neq\chi_0} } \chi(hk)\overline{\chi(g\ell)}.\]
Then, the orthogonality of Dirichlet characters implies
\begin{align}
    \label{ITeq}
    I(T) = (q-1)G(T)-E(T),
\end{align}
where
\begin{align*}
    G(T) := \frac{T}{\log T}\sum_{h,g\in\mch}r(h)r(g) \sum_{\substack{k,\ell\ge1,(k\ell,q)=1 \\ hk\equiv g\ell (\bmod q)}}\frac{\widehat{K}(\log k\ell)}{\sqrt{k\ell}}\whp\Bg(\frac{T}{\log T}\log \frac{g\ell}{hk}\Bg)
\end{align*}
and let \(E(T)\) denote the corresponding contribution arising from the second term in the character orthogonality relation. 

We first consider \(G(T)\). For each \(m,n \in \mcm\), let \(u,v \in \mbz\), and \(a,b \pmod q\) be the unique indices such that \(m \in \mcm_{u,a}\), \(n \in \mcm_{v,b}\). Then, put \(h = h_{u,a}\), \(g = h_{v,b}\). Define 
\[k=\frac{n}{(m,n)}\quad \text{and} \quad \ell = \frac{m}{(m,n)}.\]
Thus, we have \((k,\ell)=1\) and \(mk=n\ell\). Since \(h\equiv m\pmod q\), \(g\equiv n\pmod q\), and every element of \(\mcm\) is coprime to \(q\), we have 
\[(m,q)=(n,q)=((m,n),q)=1.\]
Hence the integers \(k,\ell\) satisfy \((k\ell,q)=1\) and \(hk\equiv g\ell \pmod q\). Moreover, by the definition of sets \(\mcm_{u,a}\) and \(\mcm_{v,b}\)
\[1 \le \frac{m}{h} < 1+\frac{\log T}{T}, \quad 1 \le \frac{n}{g} < 1+\frac{\log T}{T}.\]
Since \(mk=n\ell\), 
\[\frac{g\ell}{hk}=\frac{gm}{nh}=\frac{m/h}{n/g}.\]
It follows that
\[\Bg|\frac{T}{\log T}\log \frac{g\ell}{hk}\Bg| \le \frac{T}{\log T}\log \Bg(1+\frac{\log T}{T}\Bg) \le 1\]
for sufficiently large \(T\). Therefore, we obtain the following lower bound
\[\whp\Bg(\frac{T}{\log T}\log \frac{g\ell}{hk}\Bg)\gg 1.\]
Fix \(h = h_{u,a}\), \(g = h_{v,b}\) and coprime positive integers \(k,\ell\). The number of choices 
\[m^\prime \in \mcm_{u,a} \quad \text{and}\quad n^\prime \in\mcm_{v,b} \]
satisfying \(m^\prime k = n^\prime \ell\) is at most 
\[\min\{|\mcm_{u,a}|,|\mcm_{v,b}| \}=\min\{r(h)^2,r(g)^2\} \le r(h)r(g).\]
Indeed, since \((k,\ell)=1\), the relation \(m^\prime k=n^\prime\ell\) implies
\[m^\prime = \ell d \quad \text{and} \quad n^\prime =kd\]
for some positive integer \(d\). Thus either \(m^\prime\) or \(n^\prime\) uniquely determines the other. By \eqref{Kpro2}, we have
\begin{align}
    \label{Glower1}
    G(T) \gg \frac{T}{\log T}\sum_{\substack{m,n\in\mcm \\ [m,n]/(m,n) \le T^{\varepsilon/3\eta}}}\sqrt{\frac{(m,n)}{[m,n]}}.
\end{align}
Following the argument in \cite[p. 25]{delaBreteche2019PLMS}, for fixed \(m \in \mcm\), we get
\[\sum_{n \in \mcm}\Big( \frac{(m,n)}{[m,n]} \Big)^{1/3} \le \prod_{p \le y_\mcm}\Big(1+\frac{2}{p^{1/3}-1}\Big) \ll \exp \big(y_\mcm^{2/3}\big).\]
Furthermore, Rankin's trick yields that
\begin{align*}
    \sum_{\substack{m,n\in\mcm \\ [m,n]/(m,n) \le T^{\varepsilon/3\eta}}} \sqrt{\frac{(m,n)}{[m,n]}} &= \Big(\sum_{m,n\in\mcm} - \sum_{\substack{m,n\in\mcm \\ [m,n]/(m,n) > T^{\varepsilon/3\eta}}}\Big) \sqrt{\frac{(m,n)}{[m,n]}} \nonumber \\
    & \ge \sum_{m,n\in\mcm}  \sqrt{\frac{(m,n)}{[m,n]}} -T^{-\varepsilon/18\eta}S_{1/3}(\mcm) \nonumber \\
    & \gg\sum_{m,n\in\mcm}  \sqrt{\frac{(m,n)}{[m,n]}} -T^{-\varepsilon/18\eta}N \exp \big(y_\mcm^{2/3}\big).
\end{align*}
Since \(y_\mcm \ll (\log N)^{6/5}\), using \eqref{Glower1}, we have
\begin{align}
    \label{Glower2}
    G(T) \gg \frac{T}{\log T}S(\mcm).
\end{align}

On the other hand, we have \(\widehat{K}(\log k\ell) = 0\) unless \( k \ell \le T^{2\varepsilon}\), and since both \(\widehat{K}\) and \(\whp\) are bounded, we have
\[E(T) \ll \frac{T}{\log T} \Bg(\sum_{h \in \mch}r(h)\Bg)^2 \sum_{k\ell \le T^{2\varepsilon}}\frac{1}{\sqrt{k\ell}}.\]
By the Cauchy-Schwarz inequality,
\[\sum_{h \in \mch}r(h) \le N,\]
while, for sufficiently large \(X\),
\[\sum_{k\ell \le X}\frac{1}{\sqrt{k\ell}} \ll \sqrt{X}\log (2X).\]
It follows that
\[E(T) \ll \frac{T}{\log T} N^2 T^\varepsilon\log T.\]
By our choice of \(\varepsilon\), \(\kappa+\varepsilon<\theta\). Then, since \(S(\mcm) \ge N\), we get
\[E(T) = o\Bg(q\frac{T}{\log T}S(\mcm)\Bg).\]
Combining this estimate with \eqref{ITeq} and \eqref{Glower2} shows
\[I(T) \gg q\frac{T}{\log T}S(\mcm). \]
Substituting this into \eqref{max1}, we obtain
\begin{align}
    \label{max2}
    \mcz_{q,\beta}(T)^2 \gg \frac{S(\mcm)}{N}.
\end{align}

Applying Lemma \ref{lem4}, it follows from the definition of \(N\) and \eqref{max2} that
\[
\mcz_{q,\beta}(T) \ge \exp\bgg(\bg(\sqrt{2\kappa} +o(1)\bg)\sqrt{\frac{\log T \lt T}{\lo T}}\bgg)
\]
as \(T \to \infty\). Given any \(0 < c < \sqrt{2\min\{\theta,1-\beta\}}\), we choose 
\[\frac{c^2}{2}<\kappa<\min\{\theta,1-\beta\},\]
and then choose \(\varepsilon>0\) sufficiently small so that \(\kappa+\varepsilon<\theta\). Theorem \ref{thm1} follows.

\section*{Acknowledgments}
Zikang Dong was supported by the National Natural Science Foundation of China (Grant No. 1240011770). Qiyu Yang was supported by the Natural Science Foundation of Henan Province (Grant No. 252300421782) and the National Natural Science Foundation of China (Grant No. 12601011).
	
	\bibliographystyle{siam}
    \bibliography{references}
\end{document}